\documentclass{article}

\usepackage{amsmath}
\usepackage{amsthm}
\usepackage{amssymb}
\usepackage{amsfonts}
\usepackage{pifont}
\usepackage{subeqnarray}
\usepackage{bm}
\usepackage{tikz}
\usepackage{dsfont}
\usepackage{enumerate}
\usepackage{upgreek}
\usepackage{booktabs}
\usepackage{fancyhdr}
\usepackage{multirow}
\usepackage{multicol}
\usepackage{float}
\usepackage[inline]{enumitem}   
\usepackage{cases}

\usepackage[ruled,vlined,linesnumbered,inoutnumbered]{algorithm2e}

\usepackage{bbding}
\usepackage{caption}
\usepackage{todonotes}
\usepackage{makecell} 
\usepackage{mathrsfs}

\usepackage{lineno}

\usepackage[top=2.5cm, bottom=2.5cm, left=3cm, right=3cm]{geometry} 

\usepackage{graphics,graphicx,epsf,epstopdf,subfigure}
\graphicspath{{./Figures/}}
\ifpdf
\DeclareGraphicsExtensions{.eps,.pdf,.png,.jpg}
\else
\DeclareGraphicsExtensions{.eps}
\fi

\usepackage[colorlinks,
linkcolor=blue,
anchorcolor=blue,
citecolor=blue
]{hyperref}

\usepackage{xr-hyper}

\providecommand{\U}[1]{\protect\rule{.1in}{.1in}}

	\newtheorem{theorem}{Theorem}[section]

\newtheorem{lemma}{Lemma}[section]

\newtheorem{corollary}{Corollary}[section]
\newtheorem{definition}{Definition}[section]

\newtheorem{ass}{Assumption}[section]
\newtheorem{example}{Example}[section]

\numberwithin{equation}{section}

\def\x{\mathbf{x}}
\def\z{\mathbf{z}}
\def\y{\mathbf{y}}
\def\e{\mathbf{e}}
\def\d{\mathbf{d}}

\def\proj{\mathrm{\Pi}}
\def\dist{{\mathrm{Dist}}}

\def\I{\mathcal{I}}

\def\J{\mathcal{J}}

\def\R{\mathbb{R}}

\allowdisplaybreaks
\graphicspath{ {./results/} }
\definecolor{Gray}{rgb}{0.5,0.5,0.5}

\makeatletter

\newcommand{\Rmnum}[1]{\expandafter\@slowromancap\romannumeral #1@}
\makeatother

\begin{document}




\title{New Bounds for Exact Penalized  Cardinality-Constrained Optimization with Pseudonormality Conditions}

\author{Lili Pan\thanks{Department of Mathematics, Shandong University of Technology, Zibo 255049,  China  (\texttt{panlili1979@163.com}).} 
\and Huilin Xie\thanks{Department of Mathematics, Shandong University of Technology, Zibo 255049,  China  (\texttt{17864298589@163.com}).} 
\and Xianchao Xiu\thanks{School of Mechatronic Engineering and Automation, Shanghai University, Shanghai 200444, China (\texttt{e-mail: xcxiu@shu.edu.cn}). }  
\and Jiyuan Tao \thanks{Department of Mathematical Sciences, Loyola College in Maryland, Baltimore, MD 21210, United States (\texttt{e-mail: jtao@loyola.edu}). } 
}

\date{} 

\maketitle

\begin{abstract}

\noindent
Cardinality-constrained optimization (CCO) is a popular topic in sparse learning and signal recovery, yet remains challenging due to the inherent nonconvexity and discontinuity of cardinality constraints. This paper investigates the exact penalty theory for CCO problems with general equality and inequality constraints. In particular, we extend the pseudonormality condition to the cardinality-constrained framework and establish the local exact penalization without imposing Lipschitz continuity on the objective function. We further analyze both the projected subgradient method and its stochastic variant with convergence guarantees for the derived exact penalty formulation. Compared with the existing results, we give some more precise bounds of the iterate sequence and the objective function value. 
\vskip10pt


\end{abstract}

\section{Introduction}
In this paper, we consider the  cardinality-constrained optimization (CCO) problem
\begin{equation}\label{Pro1}
\begin{aligned}
\min_{\x\in \R^n}~~& f(\x)\\
\mathrm{s.t.}~~& g(\x)\leq 0, ~h(\x)=0,\\
&\|\x\|_0\leq s,
\end{aligned}
\end{equation}
where $f(\x): \R^n\rightarrow \R$ is a continuously differentiable and bounded from below function, 
$g(\x)=(g_1(\x), \cdots, g_m(\x))^{\top}: \R^n\rightarrow \R^m$ and  
$h(\x)=(h_1(\x), \cdots, h_p(\x))^{\top}: \R^n\rightarrow \R^p$ 
are continuously differentiable functions, $\|\x\|_0$ denotes the cardinality of $\x \in \R^n$ (number of nonzero entries of $\x$), and $s<n$ is a positive integer.	
Over the past decade, CCO has been widely applied in inverse problems \cite{bruckstein2009sparse,wang2013multi}, feature selection \cite{zou2018selective,hazimeh2023grouped}, portfolio selection \cite{bertsimas2022scalable,anis2025end} and network pruning \cite{he2023structured,benbaki2023fast}.
We refer the reader to \cite{tillmann2024cardinality} for a comprehensive survey.

\subsection{Literature review}

From an optimization perspective, problem \eqref{Pro1} poses significant challenges due to the presence of the cardinality constraint. Below, we outline three primary approaches to addressing this issue.
\begin{itemize}
\item 
\textit{The first approach} involves reformulating problem \eqref{Pro1} as a mixed-integer programming problem by introducing appropriate binary variables, which can then be solved directly using discrete optimization techniques \cite{bonami2009exact,burdakov2016mathematical}. However, these methods often suffer from limited computational efficiency, particularly for large-scale instances. 
By applying standard relaxations to the binary constraints,  problem \eqref{Pro1} is reformulated into a relaxed continuous nonlinear optimization problem with an orthogonality-type constraint. 
In \cite{vcervinka2016constraint}, some problem-tailored constraint qualifications were defined  
and shown how these constraint qualfications can be used to obtain suitable optimality conditions for problem \eqref{Pro1}. 
And then the authors in \cite{kanzow2021augmented} 
investigated the viability of using a standard safeguarded multiplier penalty method without any problem-tailored modifications to solve the reformulated problem. 
More recent works in this direction can be found in \cite{liang2024relaxed,kanzow2025sparse}.

\item 
\textit{The second approach} directly tackles problem \eqref{Pro1} using continuous optimization ideas, leveraging the special structure of cardinality constraints. 
We enumerate some works on the cardinality constrained problems \eqref{Pro1} with different constraints. 
For problem \eqref{Pro1} with nonnegative constraints, an iterative hard thresholding method with rigorous convergence analysis was proposed in \cite{pan2017convergent}. For problem \eqref{Pro1} with symmetric set constraints, an efficient method based on the framework in \cite{beck2013sparsity} was developed \cite{beck2015minimization}. For problem \eqref{Pro1} with cone constraints, optimality conditions were established under suitable constraint qualifications \cite{pan2017restricted}. In the linear case of problem \eqref{Pro1}, Lagrange duality and saddle point results were obtained using piecewise linear convex optimization \cite{zhao2022lagrange}. For problem \eqref{Pro1} with general constraints, second-order optimality conditions for the existence of augmented Lagrangian multipliers were provided in \cite{kan2024second}. Additionally, for the quadratic-constrained quadratic programming variant of problem \eqref{Pro1}, an efficient semi-smooth Newton method with convergence guarantees was proposed in \cite{li2025sparse}.

\item 
\textit{The third approach} employs penalty methods, which are highly effective for problems with complex constraints. 
Exact penalization refers to a scenario where the solution of the penalized problem coincides with that of the original problem without requiring the penalty parameter to tend to infinity. 
In particular, we say that  problem \eqref{Pro1} admits an exact penalty at the 
feasible point $\x^*$ if $\x^*$ is a strict local minimum of \eqref{Pro1} for every smooth function $f$, and there is a scalar $\tau>0$ such that $\x^*$ is also a local minimum of the following problem 
\begin{equation}\label{Pro2}
	\begin{aligned}
		\min_{\x\in \R^n}&~  F_{\tau}(\x):=f(\x)+{\tau}  \left(\|g^+(\x)\|_1+\|h(\x)\|_1\right) \\
		\mathrm{s.t.}  &~~\|\x\|_0\leq s,
	\end{aligned}	
\end{equation}
where $\|g^+(\x)\|_1=\sum_{i=1}^m\max\{0, g_i(\x)\}$ and 	$\|h(\x)\|_1=\sum_{j=1}^p |h_j(\x)|$. 
This concept has been extensively studied in literature. 
Exact penalization for nonlinear programming problems with Lipschitz continuous objective functions was investigated in \cite{nocedal2006numerical}, while exact penalization for problems with nonconvex and non-Lipschitz objective functions was explored in \cite{chen2016penalty,guo2018necessary}. 

The conditions for the admittance of an exact penalization for problem  \eqref{Pro1} are called constraint qualifications. 
For example, based on the relaxed constant positive linear dependence (CCOP-RCPLD) condition, it was shown in \cite{xiao2024optimality} that problem \eqref{Pro1} admits an exact penalization. Some other constraint qualifications for  problem  \eqref{Pro1} were proposed 
in \cite{vcervinka2016constraint,burdakov2016mathematical,pan2017optimality,kanzow2025sparse}. 
It was worth noting that the pseudonormality and quasinormality were not considered for  problem  \eqref{Pro1}, which  are two central conditons  within the taxonomy of the constraint qualifications for constrained  optimization problems. 
Although the pseudonormality implies the quasinormality in the  constrained  optimization,  quasinormality fails to provide 
the desired theoretical unification when the abstract constraint is not regular  and 
pseudonormality admits an insightful geometrical interpretation in contrast with quasinormality \cite{bertsekas2002pseudonormality}. 

To solve problem \eqref{Pro2}, several greedy methods have been proposed. For instance, the binary iterative hard thresholding (BIHT) method was applied to one-bit compressive sensing \cite{jacques2013robust}. Furthermore, the adaptive iterative hard thresholding (AIHT1) method was developed to solve the constrained least absolute deviation problem with cardinality constraints \cite{li2023adaptive}. Notably, stochastic methods require significantly less storage and computation compared to general subgradient methods \cite{liu2019one}, making stochastic gradient descent methods widely attractive for optimization problems with cardinality constraints, see \cite{nguyen2017linear,zhou2018efficient,shang2021efficient,li2024stochastic}. 	
\end{itemize}

\textit{This naturally raises the question:  Can we further explore the exact penalty theory for problem \eqref{Pro1} utilizing the concept of pseudonormality  for cardinality constrained optimization and improve the bounds for projected subgradient methods associated with problem \eqref{Pro2}?}

\subsection{Contributions}

In this paper, we first define two constraint qualifications for problem \eqref{Pro1} to ensure the exactness of the penalized problem \eqref{Pro2}, which are closely related to the stationarity of the original problem and the error bound of the feasible region. We then employ the projected subgradient method and the stochastic projected subgradient method to solve the penalized problem \eqref{Pro2}. It should be noted that the projected subgradient method is a classical method for minimizing a nondifferentiable convex function subject to a compact and convex set \cite{alber1998projected,auslender2009projected}, while the cardinality-constrained set is nonconvex, rendering the convergence analysis of the projected subgradient method highly challenging for such constrained problems. Besides, the convergence results of stochastic gradient descent established in \cite{nguyen2017linear,li2024stochastic} cannot be directly applied to stochastic iterative methods, as they rely on the assumptions that the objective function is constrained strongly convex (RSC) and smooth (RSS).

The contributions of this paper can be summarized as follows:
\begin{itemize}
\item[(1)] We propose the cardinality-constrained Mangasarian-Fromovitz constraint qualification (CC-MFCQ) and cardinality-constrained pseudonormality (CC-pseudonormality) for problem \eqref{Pro1}, where CC-MFCQ implies CC-pseudonormality. We rigorously prove that CC-pseudonormality is a sufficient condition for ensuring the exact penalization of problem \eqref{Pro2}. Notably, CC-pseudonormality does not imply the CCOP-RCPLD condition proposed in \cite{xiao2024optimality}.
\item[(2)] We develop the projected subgradient method and the stochastic projected subgradient method to deal with problem \eqref{Pro2}. Moreover, we provide a convergence analysis that does not require the objective function to be RSC or RSS. Compared with the results in \cite{liu2019one}, our findings yield tighter bounds on the iterative sequence and smaller corresponding objective function values, as illustrated in Table \ref{tab-3}.
\end{itemize}

\begin{table}[h]
\centering
\renewcommand{\arraystretch}{1.1}
\caption{New bounds compared with \cite{liu2019one}, where the differences are marked in \textcolor{red}{red}.}\label{tab-3}
\begin{tabular}{|c|c|}
\hline
Previous results   & Our results \\ \hline \hline 
\multirow{2}*{\small{$\|\x^{k+1}-\x^*\|^2\leq 6L_{\tau}^2\sum_{i=1}^{k}\alpha_i^2+2L_{\tau}\|\x^*\|\sum_{i=1}^{k}\alpha_i+\epsilon^2$}} & \multirow{2}*{\small{$\|\x^{k+1}-\x^*\|^2\leq \textcolor{red}{L_{\tau}^2\sum_{i=1}^{k}\alpha_i^2}+2L_{\tau}\|\x^*\|\sum_{i=1}^{k}\alpha_i+\epsilon^2$}}\\
& \\
\multirow{2}*{\small{$\underline{F}^{k}_{\tau}-F_{\tau}^*\leq \frac{3L_{\tau}^2\sum_{i=1}^{k}\alpha_i^2}{\sum_{i=1}^{k}\alpha_i}+\frac{\epsilon^2}{2\sum_{i=1}^{k}\alpha_i}
	+L_{\tau}\|\x^*\|$}}    &
\multirow{2}*{\small{$\underline{F}^{k}_{\tau}-F_{\tau}^*\leq 	\textcolor{red}{\frac{L_{\tau}^2\sum_{i=1}^{k}\alpha_i^2}{2\sum_{i=1}^{k}\alpha_i}}+\frac{\epsilon^2}{2\sum_{i=1}^{k}\alpha_i}
+L_{\tau}\|\x^*\|$}} \\
& \\ \hline
\multirow{2}*{\small{$\mathbb{E}\left[F_{\tau}\left(\frac{1}{k}\sum_{i=1}^{k}\x^i\right)-F_{\tau}^*\right]\leq {3\beta M_2}+\frac{\epsilon^2}{2\beta k} +M_1\|\x^*\|$}}   & \multirow{2}*{\small{$\mathbb{E}\left[F_{\tau}\left(\frac{1}{k}\sum_{i=1}^{k}\x^i\right)-F_{\tau}^*\right]\leq \textcolor{red}{\frac{\beta M_2}{2}}+\frac{\epsilon^2}{2\beta k} +M_1\|\x^*\|$}}  \\ 
& \\ \hline
\multirow{2}*{\small{$\mathbb{E}\left(\|\x^{k+1}-\x^*\|^2\right)\leq \frac{24M_2}{\sigma^2(k+1)}+\frac{2M_1\|\x^*\|}{\sigma}$}}   & \multirow{2}*{\small{$\mathbb{E}\left(\|\x^{k+1}-\x^*\|^2\right)\leq \textcolor{red}{\frac{4M_2}{\sigma^2(k+1)}}+\frac{2M_1\|\x^*\|}{\sigma}$}}  \\ 
& \\
\multirow{2}*{\small{$\mathbb{E}\left[F_{\tau}\left(\sum_{i=1}^{k}\frac{2i}{k(k+1)}\x^i\right)-F_{\tau}^*\right]\leq \frac{12M_2}{\sigma(k+1)}+M_1\|\x^*\|$}}   & \multirow{2}*{\small{$\mathbb{E}\left[F_{\tau}\left(\sum_{i=1}^{k}\frac{2i}{k(k+1)}\x^i\right)-F_{\tau}^*\right]\leq \textcolor{red}{\frac{2M_2}{\sigma(k+1)}}+M_1\|\x^*\|$}}  \\ 
& \\
\hline
\end{tabular}
\end{table}

\subsection{Organization}	 

The structure of this paper is outlined as follows. Section \ref{Section2} reviews notations and basic properties. Section \ref{Section3} presents the exact penalty theory for problem \eqref{Pro1}. Section \ref{Section4} provides the convergence results for the projected subgradient method and corresponding stochastic projected subgradient method. 
Section \ref{Section6} concludes this paper.

\section{Preliminaries}\label{Section2}

\subsection{Notations}	

For a vector $\x\in \R^n$, the support set of $\x$ is denoted as $\Gamma(\x) := \{i\in\{1, \cdots,n\}\mid  x_i\neq 0\}$. Then $\|\x\|_0$ is the cardinality of $\Gamma(\x)$, or $\|\x\|_0=|\Gamma(\x)|$. For an index set $J\subseteq \{1, \cdots, n\}$, let $\x_J\in \R^n$ be the vector equal to $\x$ on $J$ and zero elsewhere, and let $\nabla_J f(\x):=(\nabla f(\x))_J$. Besides,
$\overline{J}$ is the complement of $J$ in $\{1,\cdots,n\}$. Denote the feasible region of problem \eqref{Pro1} as $\Omega\cap S$ with $\Omega := \{\x\in \R^n\mid g(\x)\leq 0, ~h(\x)=0\}$ and $S := \{\x\in \R^n \mid  \|\x\|_0\leq s\}$. Define $\J_s := \{J\subseteq \{1,\cdots,n\}\mid |J|\leq s\}$, then $S$ admits the decomposition $S=\cup_{J\in \J_s}S_J$, where $S_J := \{\x\in \R^n\mid \Gamma(\x)\subseteq J\}$. For $\x^*\in \Omega\cap S$, denote $\I(\x^*):= \{i\in\{1,\cdots,m\}\mid g_i(\x^*)=0\}$ and $\J(\x^*):=\big\{J\subseteq \{1,\cdots,n\}\mid \Gamma(\x^*)\subseteq J\in\J_s \big\}$. For a closed set $A$, the projection of $\x\in \R^n$ onto $A$ is defined as $\proj_A(\x):={\rm argmin}_{\y\in A} ~\|\y-\x\|^2$, and the distance from $\x$ to  $A$ is $\dist(\x,A):=\|\y-\x\|$ with $\y\in \proj_A(\x)$. 
In particular, the projection $\proj_S(\x)$  sets all but the largest (in magnitude) $s$ elements of $\x$ to zero.

\subsection{Tangent and normal cones}	
Let $A\subseteq \R^{n}$ be an arbitrary closed set and  $\x^*\in A$. According to \cite{rockafellar1998variational}, the Bouligand (also contingent or tangent) cone ${T}_{A}(\x^*)$ to set $A$ at $\x^*$ is defined as
\begin{equation*}
	T_{A}(\x^*):=\left\{\mathbf{d}\in\mathbb{R}^{n}\mid \exists \{\x^k\}\xrightarrow{A} \x^* ~\mathrm{and}~ \{t_k\}\downarrow 0
	~\mathrm{such~that} \underset{k\rightarrow \infty}{\lim}\frac{\x^k-\x^*}{t_k}=\mathbf{d}~\right\},
\end{equation*}
	where $\{\x^k\}\xrightarrow{A} \x^*$ means that $\{\x^k\} \subseteq A$ and $\{\x^k\}$ converges to
	$\x^*$.
The Fr$\acute{e}$chet (also regular) normal cone is defined as the polar of the Bouligand tangent cone, i.e.,
\begin{equation}\nonumber
			\widehat{N}_{A}(\x^*):=T_{A}(\x^*)^{\circ}=\big\{\mathbf{v}\in \R^n \mid \langle \mathbf{v},\mathbf{d}\rangle\leq 0, ~\mathbf{d}\in T_{A}(\x^*)\big\}.		
	\end{equation}
The Mordukhovich (also limiting or basic) normal cone is given by
	\begin{equation}\nonumber
		N_{A}(\x^*):= \big\{ \mathbf{d}\in \R^n\mid  \exists  \{\x^k\}\xrightarrow{A} \x^*,~ \mathbf{d}^k\in\widehat{N}_{A}(\x^k) ~\mathrm{such~that}~ \mathbf{d}^k\to \mathbf{d}\big\}.
	\end{equation}
    From \cite{pan2017restricted,pan2017optimality}, the Bouligand tangent cone, Fr$\acute{e}$chet normal cone, and Mordukhovich normal cone for $S$ take the form
	\begin{equation}\nonumber
		\begin{aligned}
			{T}_{S}(\x^*) &=\left\{\begin{array}{*{20}c}
				S_{\Gamma^*}, & \textup{if}~~|\Gamma^*|=s,\\
				\underset{J\in \mathcal{J}^*}{\bigcup}S_{J}, & \textup{if}~~|\Gamma^*|<s,
			\end{array}\right.\\
			\widehat{N}_{S}(\x^*)&=\left\{\begin{array}{*{20}c}
				S_{\overline{\Gamma}^*}, & \textup{if}~~|\Gamma^*|=s, \\
				\{0\}, & \textup{if}~~|\Gamma^*|<s,\end{array}\right.\\
			N_{S}(\x^*)&= \left\{\begin{array}{cl}
				S_{\overline{\Gamma}^*}, & \textup{if}~~|\Gamma^*|=s,\\
				\underset{J\in \mathcal{J}^*}{\bigcup}S_{\overline{J}}, & \textup{if}~~|\Gamma^*|<s,
			\end{array}\right.
		\end{aligned}
	\end{equation}
	where $\Gamma^*:=\Gamma(\x^*)$. The linearized cone of $\Omega$ at $\x^*\in \Omega$ is
	\begin{equation*}
		L_{\Omega}(\x^*) := \big\{\d\in \R^n\mid \langle\nabla g_i(\x^*),\d \rangle\leq 0, ~\forall i\in \I(\x^*), ~\langle\nabla h_j(\x^*),\d \rangle= 0,~\forall j=1, \cdots,p\big\}
	\end{equation*}
	and its polar cone  is
\begin{equation*}
L_{\Omega}(\x^*)^{\circ}:=\left\{\sum_{i\in \I(\x^*)}\lambda_i\nabla g_i(\x^*)
	+\sum_{i=1}^p \mu_j\nabla h_j(\x^*)\mid \lambda_i
	\geq 0,~\forall i\in \I(\x^*),~\mu_j\in \R,~j=1, \cdots, p\right\}. 	
    \end{equation*}

\subsection{Constraint qualifications}	

Consider the following nonlinear programming problem
	\begin{equation}\label{NP}
			\min~ f(\x) \quad \mathrm{s.t.} \quad  \x \in \Omega.
	\end{equation}
	
	\begin{definition}\label{def2.1}
		We say that $\x^*\in \Omega$ satisfies  Mangasarian-Fromovitz constraint qualification  (MFCQ) if  $\nabla h_1(\x^*),\cdots, \nabla h_p(\x^*)$ are linearly independent and there exists $\d\in \R^n$ such that
			\begin{equation*}
            \langle \nabla g_i(\x^*),\d\rangle<0,~\forall i\in \I(\x^*),~\langle \nabla h_j(\x^*),\d\rangle=0,~j=1,\cdots,p.
                \end{equation*}
	\end{definition}

 	\begin{definition}\label{def2.2}
		We say that $\x^*\in \Omega$ is pseudonormal if there do not exist  $\lambda_1,\cdots,\lambda_m,\mu_1,\cdots,\mu_p$ and $\{\x^k\}\subseteq \Omega$ such that
			\begin{itemize}
				\item[(1)] $-\left[\sum_{i=1}^m\lambda_i\nabla g_i(\x^*)+\sum_{j=1}^p\mu_j\nabla h_j(\x^*)\right]=0$;
				\item[(2)]  $\lambda_i\geq 0$  for all $i=1,\cdots,m$, and $\lambda_i=0$  for all $i\notin \I(\x^*)$;
				\item[(3)]  $\{\x^k\}\to \x^*$ and $\sum_{i=1}^m\lambda_ig_i(\x^k)+\sum_{j=1}^p\mu_jh_j(\x^k)>0$, $\forall k$.
			\end{itemize}
	\end{definition}

\begin{definition}\label{def2.3}
		We say that $\x^*\in \Omega$ satisfies the Abadie CQ if $T_{\Omega}(\x^*)=L_{\Omega}(\x^*)$.
	\end{definition}

Following \cite{andreani2012relaxed,guo2014new}, the relationships among the above three CQs are 
\begin{equation*}	\mathrm{MFCQ}\quad\Rightarrow\quad\mathrm{pseudonormality}\quad\Rightarrow\quad\mathrm{Abadie ~CQ}.
\end{equation*}

\subsection{Optimality conditions}	

Based on the above MFCQ, the following Karush-Kuhn-Tucker (KKT) conditions hold.
\begin{lemma} If the MFCQ holds at $\x^*\in \Omega$, then $\x^*$ is a KKT point of problem \eqref{NP}, that is,
\begin{equation*}	
-\nabla f(\x^*)\in L_{\Omega}(\x^*)^{\circ}. 
\end{equation*}
\end{lemma}

Furthermore, inspired by the pseudonormality, it is not hard to obtain the exact penalization theory.

\begin{lemma} If $\x^*\in \Omega$ is pseudonormal, then $\Omega$ admits an exact penalty at $\x^*$. That is, if $\x^*$ is a strictly local minimum of the smooth function $f$ over $\Omega$, there exists $\tau>0$ such that $\x^*$ is also a local minimum of  $F_{\tau}(\x)$. 
\end{lemma}

See \cite[Proposition 4.2]{bertsekas2002pseudonormality} for more details.

    
\section{Exact penalization for problem \eqref{Pro1}}\label{Section3}

A variety of CQs for the nonlinear programming problem \eqref{NP} have been extended to analyze the corresponding CCO problem. In \cite{xiao2024optimality}, it was shown that the RCPLD condition provides a sufficient condition for error bounds of problem \eqref{Pro1}, which in turn implies exact penalization when the objective function is Lipschitz continuous. In addition, the relationships between RCPLD and other CQs were discussed and established.

\subsection{Constraint qualifications}

\begin{definition}\label{def3.1}
We say that the cardinality constrained Mangasarian-Fromovitz constraint qualification (CC-MFCQ) holds at $\x^*\in \Omega \cap S$, if for each $J\in \J(\x^*)$, the gradients $\nabla_J h_1(\x^*),\cdots,\nabla_J h_p(\x^*)$ are linearly independent, there exists $\d\in S_J$ such that
\begin{equation*}	
\langle \nabla g_i(\x^*),\d\rangle<0, ~i\in \I(\x^*), ~\langle \nabla h_j(\x^*),\d\rangle=0,~j=1,\cdots,p.
\end{equation*}
\end{definition}

From Motzkin's theorem, for each $J\in \J(\x^*)$, $\nabla_J g_i(\x^*)$, $i\in \I(\x^*), \nabla_J h_1(\x^*),\cdots,\nabla_J h_p(\x^*)$  are positively independent. That is, if
\begin{equation*}	
\sum_{i\in \I(\x^*)}\lambda_i\nabla_J g_i(\x^*)+ \sum_{j=1}^p\mu_j\nabla_J h_j(\x^*)= 0
\end{equation*}
and $\lambda_i\geq 0$, $i\in \I(\x^*)$ and $\mu_j\in \R$, $j=1,\cdots,p$, then $\lambda_i=0$, $i\in \I(\x^*)$ and $\mu_j=0$, $j=1,\cdots,p$.
Furthermore, for each $J\in \J(\x^*)$, $\nabla g_i(\x^*)$,  $i\in \I(\x^*)$, $\nabla h_1(\x^*),\cdots,\nabla h_p(\x^*)$  and $\e_k$, $k\in \overline{J}$  are positively independent. Specifically, if
\begin{equation*}	
\sum_{i\in \I(\x^*)}\lambda_i\nabla g_i(\x^*)+ \sum_{j=1}^p\mu_j\nabla_J h_j(\x^*)+\sum_{k\in \overline{J}} \gamma_k\e_k= 0,
\end{equation*}
where $\lambda_i\geq 0$, $i\in \I(\x^*)$,  $\mu_j\in \R$, $j=1,\cdots,p$ and $\gamma_k\in \R$, $k\in \overline{J}$, then $\lambda_i=0$, $i\in \I(\x^*)$ and $\mu_j=0$, $j=1,\cdots,p$ and $\gamma_k=0$, $k\in \overline{J}$, where $\e_k$ denotes the $k$-th unit vector in $\R^n$.

The condition introduced in \cite{vcervinka2016constraint} (also called restricted MFCQ in \cite{pan2017optimality} or restricted Robinson CQ in \cite{pan2017restricted}) can be expressed as  $\nabla_{\Gamma(\x^*)} g_i(\x^*)$ $(i\in \I(\x^*))$ are positively independent. It is straightforward to verify that these conditions imply the CC-MFCQ defined above, but not vice versa as showed in the following example.

\begin{example}
Let $\x^*=(1,1,0,0)^{\top}$, $s=3$ and 
\begin{equation*}	
\Omega=\{\x\in \R^4\mid g_1(\x)=x_1+x_3+x_4-1\leq 0,~g_2(\x)=x_3+x_4 \leq 0, ~
		g_3(\x)=x_2-1\leq 0\}.
\end{equation*}
\end{example}

By simple calculation, $\Gamma(\x^*)=\{1,2\}$, $\I(\x^*)=\{1,2,3\}$,  $\J(\x^*)=\{\{1,2,3\},\{1,2,4\}\}$, and
\begin{equation*}
	\nabla g_1(\x^*)=\left(\begin{array}{l} 1 \\0\\1\\1\end{array}\right),~
	\nabla g_2(\x^*)=\left(\begin{array}{l} 0 \\0\\1\\1\end{array}\right),~
	\nabla g_3(\x^*)=\left(\begin{array}{l} 0 \\1\\0\\0\end{array}\right).
\end{equation*}
One can check that  $\nabla_{\Gamma(\x^*)} g_i(\x^*)$ ($i\in \I(\x^*)$) is positively dependent, so restricted MFCQ fails. However, $\nabla_{J} g_i(\x^*)$ ($i\in \I(\x^*)$) for any $J\in \J(\x^*)$ is positively independent, and thus CC-MFCQ holds.

Let $\x^*\in \Omega\cap S$ satisfy  CC-MFCQ, from \cite[Theorem 6.14]{rockafellar1998variational}, it holds
\begin{equation}
\label{M-normal-cap} N_{\Omega\cap S}(x^*)\subseteq N_{\Omega}(\x^*)+N_{S}(\x^*)=L_{\Omega}(\x^*)^{\circ}+N_{S}(\x^*),
\end{equation}
where the equality dues to the fact that CC-MFCQ of $\Omega\cap S$ results in MFCQ of $\Omega$.
Using \eqref{M-normal-cap} and \cite[Theorem 6.12]{rockafellar1998variational}, we obtain the following optimality condition for problem \eqref{Pro1}.
\begin{theorem}
Suppose that $\x^*\in \Omega\cap S$ is a local minimizer of problem \eqref{Pro1} and CC-MFCQ holds at $\x^*$. Then $\x^*$  is a KKT  point of problem \eqref{Pro1}, i.e.,
there exist $\lambda^*\in \R^m$ and $\mu^*\in \R^p$ such that
\begin{eqnarray*} \label{BCMKKT}
\begin{cases}
-\nabla f(\x^*)-\sum_{i=1}^m \lambda_i^* \nabla g_i(\x^*)-\sum_{j=1}^p \mu_j^* \nabla h_j(\x^*)\in \widehat{N}_S(\x^*), \\
\lambda^*_i\geq 0,~g_i(\x^*)\leq 0,~ \lambda^*_ig_i(\x^*)=0,~i=1,\cdots,m,\\ 
\|\x^*\|_0\leq s. 
\end{cases}
\end{eqnarray*}
\end{theorem}

Now, we extend pseudonormality  in Definition \ref{def2.2}  to problem \eqref{Pro1}.

\begin{definition}\label{R-MF}
We say that $\x^*$ is CC-pseudonormal if here do not exist $\lambda_1,\cdots,\lambda_m,\mu_1, \cdots,\mu_p$ and sequence $\{\x^k\}\subseteq S$ such that
		\begin{itemize}
			\item[(1)] $-\left[\sum_{i=1}^m\lambda_i\nabla g_i(\x^*)+\sum_{j=1}^p\mu_j\nabla h_j(\x^*)\right]\in N_S(\x^*)$;
			\item[(2)] $\lambda_i\geq 0$ for all $i=1,\cdots,m$, and $\lambda_i=0$ for all $i\notin \I(\x^*)$;
			\item[(3)] $\{\x^k\}\to \x^*$ and $\sum_{i=1}^m\lambda_ig_i(\x^k)+\sum_{j=1}^p\mu_jh_j(\x^k)>0$, $\forall k$.
		\end{itemize}
	\end{definition}

\subsection{Theoretical guarantee}

In \cite{xiao2024optimality}, local error boundedness and local exact penalization for problem \eqref{Pro1} were established under RCPLD and Lipschitz continuity of the objective function. Notice that pseudonormality does not imply RCPLD \cite{andreani2012relaxed} in nonlinear programming problems. 
In this regard, we prove local exact penalization under CC-pseudonormality without requiring Lipschitz continuity.

\begin{theorem}
Assume that $\x^*\in \Omega\cap S$ is CC-pseudonormal and is a strictly local minimizer of problem \eqref{Pro1}. Then there exists $\tau^*>0$ such that $\x^*$ is also a local minimizer of problem \eqref{Pro2} for all   $\tau>\tau^*$.
\end{theorem}

\begin{proof}
Suppose for contradiction that $\x^*$ is a strict local minimizer of problem \eqref{Pro1} but not a local minimizer of
\begin{equation}\label{SCP1}
\min~~f(\x)+k[\|g^+(\x)\|_1+\|h(\x)\|_1] \quad \mathrm{s.t.}\quad \x\in S
\end{equation}
for all $k=1,2, \cdots$. Let $\epsilon>0$ and $\x^*$ is a strict minimizer of
\begin{equation}\label{SICO1}
\min ~~f(\x)\quad \mathrm{s.t.} \quad \x\in \Omega\cap S, ~\|\x-\x^*\|\leq \epsilon.
\end{equation}
Suppose that $\x^k\in S$ is the minimizer of problem  \eqref{SCP1} in the neighborhood $N(\x^*,\epsilon)$. As $\x^*$ is not a local minimizer of problem \eqref{SCP1}, we can see that $\x^k\neq \x^*$, which derives that $\x^k\notin \Omega$.
Otherwise, if $\x^k\in \Omega\cap S$, or $\|g^+(\x^k)\|_1=\|h(\x^k)\|_1=0$, it implies that
\begin{equation*}
f(\x^k)= f(\x^k)+k[\|g^+(\x^k)\|_1+\|h(\x^k)\|_1]\leq f(\x^*)+k[\|g^+(\x^*)\|_1+\|h(\x^*)\|_1]=f(\x^*),
\end{equation*}
which contradicts with $\x^*$ being a strictly local minimizer of problem \eqref{Pro1}.
Thus $\|g^+(\x^k)\|_1+\|h(\x^k)\|_1>0$ and
\begin{equation}\label{eq1}
f(\x^k)+k[\|g^+(\x^k)\|_1+\|h(\x^k)\|_1]\leq f(\x^*).
\end{equation}
It follows that  $\|g^+(\x^k)\|_1+\|h(\x^k)\|_1\to 0$ as $k\to \infty$. Since $\{\x^k\}$ is bounded, let $\overline{\x}$ be an accumulation point. Then $\overline{\x}\in \Omega\cap S$ and \eqref{eq1} gives $f(\overline{\x})\leq f(\x^*)$. Combining this with the fact that $\x^*$ being a strict  minimizer of problem \eqref{SICO1}, it is not hard to show that $\overline{\x}=\x^*$. Without loss of generality, assume that $\x^k$ converges to $\x^*$. Then for sufficiently large $k$, it holds $\|\x^k-\x^*\|< \epsilon$. This leads to the the following optimality condition
\begin{equation}\label{eq2}
-\left[\frac{1}{k} \nabla f(\x^k)+\sum_{i=1}^m \lambda_i^k\nabla g_i(\x^k)
+\sum_{j=1}^p \mu_i^k\nabla h_j(\x^k) \right]\in \widehat{N}_S(\x^k),
\end{equation}
where
\begin{equation}\label{lambdamu}
			\lambda_i^k \left\{
			\begin{array}{lc} =1,&g_i(\x^k)>0,\\
				=0,& g_i(\x^k)<0,\\ \in [0,1],&g_i(\x^k)=0, \end{array}
			\right.\qquad \textrm{and} \qquad 
			\mu_j^k \left\{
			\begin{array}{lc} =1,&h_j(\x^k)>0,\\
				=-1,& h_j(\x^k)<0,\\ \in [-1,1],&h_j(\x^k)=0. \end{array}
			\right.
\end{equation}
Since $\|g^+(\x^k)\|_1+\|h(\x^k)\|_1>0$, there exists a subsequence $\{\lambda^k,\mu^k\}_{k\in \mathcal{K}}$ such that for all $k\in \mathcal{K}$, $\lambda_i^k=1$ and $g_i(\x^k)>0$ for some index $i\in \{1,\cdots,m\}$ and $|\mu_j^k|=1$ and $h_j(\x^k)\neq 0$ for some index $j\in \{1,\cdots,p\}$. Thus
\begin{equation*}
\sum_{i=1}^m\lambda_ig_i(\x^k)+\sum_{j=1}^p\mu_jh_j(\x^k)>0, ~\forall k\in \mathcal{K}.
\end{equation*}
Let $(\lambda^*,\mu^*)$ be a limit point of this subsequence of $\{(\lambda^k,\mu^k)\}$.
Then it holds that $(\lambda^*,\mu^*)\neq 0$ and $\lambda^*\geq 0$. 
Passing to the limit in \eqref{eq2} as $k\to \infty$ and using $ \widehat{N}_S(\x^k)\to N_S(\x^*) $, we obtain
\begin{equation*}\label{eq3}
-\sum_{i=1}^m \lambda_i^*\nabla g_i(\x^*)-\sum_{j=1}^p \mu_j^*\nabla h_j(\x^*)\in N_S(\x^*).
\end{equation*}
This contradicts the CC-pseudonormality of $\Omega$ at $\x^*$. The proof is completed.
\end{proof}

\section{New bounds for convergent methods}\label{Section4}
	
In this section, we employ the projected subgradient method (PSM) and the corresponding stochastic projected subgradient method (SPSM) to solve problem \eqref{Pro2}. Furthermore, we derive improved bounds for the convergence analysis.

\subsection{Projected subgradient method}		

The projected subgradient method was adopted in \cite{liu2019one} to solve problems with cardinality constraints, where the sequence was shown to be bounded, the objective function sequence non-increasing, and convergent to the minimal value. Here, we also apply the projected subgradient method to solve problem \eqref{Pro2}, as summarized in Algorithm \ref{PSM}. Under the same condition that the objective function of problem \eqref{Pro2} is Lipschitz continuous, we establish a tighter bound for the generated sequence than those provided in \cite{liu2019one}. Furthermore, we prove that the sequence by the PSM with Polyak stepsizes as in \cite{allen1987generalization} for problems with convex constraints is convergent. 

\begin{algorithm}[ht!]
	\caption{Projected subgradient method (PSM) for problem \eqref{Pro2}} \label{PSM}	
	\KwIn{Given $\x^0=0$, let $k= 0$.}
	\While{not converged}
	{
	(\textit{Step 1}) Pick $\x^{k+1}=\proj_S(\x^k-\alpha_k \d^k)$, where  $\alpha_k$ is the stepsize  and  $\d^k\in \partial F_{\tau}(\x^k)$, that is 
            \begin{equation*}
            \d^k=\nabla f(\x^k)+ {\tau}\sum_{i=1}^{m}\lambda_i^k\nabla g_i(\x^k)+
				{\tau}\sum_{j=1}^{p}\mu_j^k\nabla h_j(\x^k)
                \end{equation*}
				with $\lambda_i^k$ and $\mu_i^k$ are defined as in \eqref{lambdamu}.
                
   (\textit{Step 2}) Check convergence: if 
\begin{equation*}
\|\x^{k+1}-\x^k\|\leq \epsilon,
 \end{equation*}  
 then stop. Otherwise, let $k \leftarrow  k+1$ and  go to \textit{Step 1}.		
	}	
	\KwOut{$\mathbf{x}$.}	
\end{algorithm}

Unlike the ordinary gradient method for differentiable objective functions, the subgradient method is not a descent method. Furthermore, the stepsizes in the subgradient method are fixed in advance \cite{boyd2003subgradient}, rather than being selected via a line search as in the ordinary gradient method. The following types of stepsize rules are commonly used
\begin{equation}\label{stepsize1}\alpha_k=\alpha;\end{equation}
\begin{equation}\label{stepsize2}\alpha_k=s/\|\d^k\|;\end{equation}
\begin{equation}\label{stepsize3}
\alpha_k\geq 0,~\sum_{i=1}^{\infty}\alpha_k=\infty,~\sum_{i=1}^{\infty}\alpha_k^2<\infty; 
\end{equation}
\begin{equation}\label{stepsize4}\alpha_k=s_k(F_{\tau}^k-\rho)/\|\d^k\|^2, \end{equation}
where  $\x^*$ is an optimal solution of problem \eqref{Pro2} and $F_{\tau}^*=F_{\tau}(\x^*)$  satisfying 
\begin{equation}\label{stepsize41}
\rho\geq F_{\tau}^*+L_{\tau}\|\x^*\|, ~F_{\tau}^k>\rho, ~0<s_k\leq \beta<2, ~ \sum_{i=1}^{\infty}s_i=\infty 
\end{equation}
and $\beta$ is  a fixed constant \cite{allen1987generalization}.	

Throughout this section, we make the following assumption for problem \eqref{Pro1}. 
\begin{ass}\label{ass3}
Functions $f$, $g_i~(i=1,\cdots,m)$, and $h_j~(j=1,\cdots,p)$ are Lipschitz continuous and convex on $\R^n$. 
\end{ass}
If Assumption \ref{ass3} holds, then there exists a constant $L_{\tau}>0$ such that 
\begin{equation*}
\| F_{\tau}(\y)-F_{\tau}(\x)\|\leq L_{\tau}\|\y-\x\|, \ \forall \ \x,\y\in \R^n,
\end{equation*}		
which is equivalent to $\|\d\|\leq L_{\tau}$  for each $\d\in \partial F_{\tau}(\x)$ and $\x\in \R^n$. It is easy to see that if functions $f$, $g_i~(i=1,\cdots,m)$, and $h_j~(j=1,\cdots,p)$  are convex, then the function $F_{\tau}$ is also convex. 
	
In order to prove the convergence of PSM, we need the following lemma. This lemma makes us give  tighter bounds of the generating sequence than the bounds in \cite{liu2019one} and the objective values are smaller at the same time. 
\begin{lemma}
Let sequence $\{\x^k\}$ be generated by PSM and denote $\z^k:=\x^k-\alpha_k \d^k$. Then for any $k\geq 1$, we have 
		\begin{equation}\label{eqle3-1}
			\begin{aligned}
				\|\x^{k+1}-\x^*\|^2&=\|\x^k-\x^*\|^2+\alpha_k^2\|\d^k\|^2-\|\x^{k+1}-\z^k\|^2\\
				&~~~~-2\langle\x^{k+1}-\z^k,\x^* \rangle +2\alpha_k \langle \d^k,\x^*-\x^k \rangle.
			\end{aligned}
		\end{equation}
\end{lemma}

\begin{proof}
		From  $\x^{k+1}=\proj_S(\z^k)$ and the property of projection onto sparse set, 
		if we denote $\Gamma^{k+1}=\Gamma(\x^{k+1})$, then 
		$\x^{k+1}=((\x^{k+1})_{\Gamma^{k+1}},{\bf{0}})=((\z^{k})_{\Gamma^{k+1}},\bf{0})$, which 
		yields that 
		$\|\x^{k+1}-\z^k\|^2=\|(\z^k)_{\overline{\Gamma}^{k+1}}^{k+1}\|^2$. 
		Then it holds that 
		\begin{equation}\label{projS1}
			\langle \x^{k+1}-\z^k,\z^k\rangle=-\|\x^{k+1}-\z^k\|^2.
		\end{equation}
		We can deduce that  
		\begin{equation*}
			\begin{aligned}
				&\|\x^{k+1}-\x^*\|^2\\
				&=\|\x^{k+1}-\x^k+\x^k-\x^*\|^2\\
				&=\|\x^k-\x^*\|^2+\|\x^{k+1}-\x^k\|^2+2\langle\x^{k+1}-\x^k, \x^k-\x^*\rangle\\
				&=\|\x^k-\x^*\|^2+\|\x^{k+1}-\z^k+\z^k-\x^k\|^2+2\langle\x^{k+1}-\z^k+\z^k-\x^k, \x^k-\x^*\rangle\\
				&=\|\x^k-\x^*\|^2+\|\x^{k+1}-\z^k\|^2+\alpha_k^2\|\d^k\|^2+2\langle \x^{k+1}-\z^k,\z^k-\x^k\rangle\\
				&~~~~+2\langle\x^{k+1}-\z^k, \x^k-\x^*\rangle+2\langle\z^k-\x^k, \x^k-\x^*\rangle\\
				&=\|\x^k-\x^*\|^2+\|\x^{k+1}-\z^k\|^2+\alpha_k^2\|\d^k\|^2+2\langle \x^{k+1}-\z^k,\z^k\rangle
				-2\langle \x^{k+1}-\z^k,\x^k\rangle\\
				&~~~~+2\langle\x^{k+1}-\z^k, \x^k\rangle-2\langle\x^{k+1}-\z^k, \x^*\rangle-2\langle\alpha_k \d^k, \x^k-\x^*\rangle\\
				&\makecell[c]{\overset{\eqref{projS1}}{=}}
				\|\x^k-\x^*\|^2+\|\x^{k+1}-\z^k\|^2+\alpha_k^2\|\d^k\|^2-2\|\x^{k+1}-\z^k\|^2\\
				&~~~~-2\langle\x^{k+1}-\z^k, \x^*\rangle+2\alpha_k\langle \d^k, \x^*-\x^k\rangle\\
				&=\|\x^k-\x^*\|^2+\alpha_k^2\|\d^k\|^2-\|\x^{k+1}-\z^k\|^2-2\langle\x^{k+1}-\z^k, \x^*\rangle+2\alpha_k\langle \d^k, \x^*-\x^k\rangle.
			\end{aligned}
		\end{equation*}
		The proof is finished.
\end{proof}
The next theorem gives the convergence of PSM. Compared with the convergence in \cite{liu2019one}, we give a smaller upper bound of the difference between the generating sequence $\x^k$ and the optimal point $\x^*$ of problem \eqref{Pro2}. Specifically, it was show in \cite{liu2019one} that  
\begin{equation*}
\|\x^{k+1}-\x^*\|^2\leq\epsilon^2+6L_{\tau}^2\sum_{i=1}^{k}\alpha_i^2+2L_{\tau}\|\x^*\|\sum_{i=1}^{k}\alpha_i,
\end{equation*}
where $\epsilon$ is a given constant. And 
\begin{equation*}
\underline{F}^{k}_{\tau}-F_{\tau}^*\leq \frac{\epsilon^2}{2\sum_{i=1}^{k}\alpha_k} +\frac{3L_{\tau}^2\sum_{i=1}^{k}\alpha_i^2}{\sum_{i=1}^{k}\alpha_k}+L_{\tau}\|\x^*\|,
\end{equation*}
where 
\begin{equation}\label{eq4-50}
\underline{F}^{k}_{\tau}:=\min\{F_{\tau}(\x^1),\cdots,F_{\tau}(\x^k)\}. 
\end{equation}
Under the same conditions, we have the following results. 
\begin{theorem}\label{th3-1}
Suppose that Assumption \ref{ass3} holds and  $\x^*$ is an optimal solution of problem \eqref{Pro2}. If  the initial point $\x^1$ satisfies $\|\x^1-\x^*\|\leq \epsilon$ for some $\epsilon>0$, then the sequence $\{\x^k\}$ by PSM satisfies
\begin{equation*}
\|\x^{k+1}-\x^*\|^2\leq\epsilon^2+L_{\tau}^2\sum_{i=1}^{k}\alpha_i^2+2L_{\tau}\|\x^*\|\sum_{i=1}^{k}\alpha_i. 
\end{equation*}
Moreover, it holds that 
\begin{equation}\label{eq3-42}
\underline{F}^{k}_{\tau}-F_{\tau}^*\leq \frac{\epsilon^2}{2\sum_{i=1}^{k}\alpha_k}
+\frac{L_{\tau}^2\sum_{i=1}^{k}\alpha_i^2}{2\sum_{i=1}^{k}\alpha_k}+L_{\tau}\|\x^*\|,
\end{equation}
where $\underline{F}^{k}_{\tau}$ is defined in \eqref{eq4-50} and $F_{\tau}^*:=F_{\tau}(\x^*)$.
\end{theorem}
    
\begin{proof}
It follows from the convexity of $F_{\tau}$ and $\d^k\in \partial F_{\tau}(\x^k)$ that 
\begin{equation}\label{subdiff}
F_{\tau}(\x^*)\geq F_{\tau}(\x^k)+ \langle \d^k,\x^*-\x^k \rangle. 
\end{equation}
Together with \eqref{eqle3-1} and \eqref{subdiff}, we have 
\begin{align}
				&\|\x^{k+1}-\x^*\|^2 \nonumber \\
				&\makecell[c]{\overset{\eqref{eqle3-1}}{\leq}}\|\x^k-\x^*\|^2+\alpha_k^2\|\d^k\|^2+2\|\x^{k+1}-\z^k\|\|\x^*\| +2\alpha_k \langle \d^k,\x^*-\x^k \rangle \nonumber \\
				&\leq\|\x^k-\x^*\|^2+\alpha_k^2\|\d^k\|^2+2\|\x^k-\z^k\|\|\x^*\| +2\alpha_k\langle \d^k,\x^*-\x^k \rangle \nonumber \\	
				&=\|\x^k-\x^*\|^2+\alpha_k^2\|\d^k\|^2+2\alpha_k\|\d^k\|\|\x^*\|+2\alpha_k\langle \d^k,\x^*-\x^k \rangle \label{le3-45} \\
				&\leq\|\x^k-\x^*\|^2+\alpha_k^2L_{\tau}^2+2\alpha_kL_{\tau}\|\x^*\|+2\alpha_k\langle \d^k,\x^*-\x^k \rangle \label{le3-21} \\
				&\makecell[c]{\overset{\eqref{subdiff}}{\leq}}\|\x^k-\x^*\|^2+\alpha_k^2L_{\tau}^2+2\alpha_kL_{\tau}\|\x^*\|+2\alpha_k(F_{\tau}^*-F_{\tau}^k),  \label{le3-42}
			\end{align}
where the second inequality comes from $\x^{k+1}=\proj_S(\z^k)$. 
This yields that 
\begin{equation}\label{eq3-30}
\|\x^{k+1}-\x^*\|^2\leq\|\x^k-\x^*\|^2+\alpha_k^2L_{\tau}^2+2\alpha_kL_{\tau}\|\x^*\|.  	
\end{equation}
By induction on $k$, we can see the first assersion holds. 
		
Furthermore, it follows from \eqref{le3-42} that 
\begin{equation*}
2\alpha_k(F_{\tau}^k-F_{\tau}^*)\leq \|\x^k-\x^*\|^2-\|\x^{k+1}-\x^*\|^2+\alpha_k^2L_{\tau}^2+2\alpha_kL_{\tau}\|\x^*\|.
\end{equation*}
Summing the above inequality from $1$ to $k$, we have 
\begin{equation*}
			\begin{aligned}
				&2\sum_{i=1}^{k}\alpha_i(F_{\tau}^i-F_{\tau}^*)\\
				&\leq\|\x^1-\x^*\|^2-\|\x^{k+1}-\x^*\|^2+L_{\tau}^2\sum_{i=1}^{k}\alpha_i^2+2L_{\tau}\|\x^*\|\sum_{i=1}^{k}\alpha_i\\
				&\leq\|\x^1-\x^*\|^2+L_{\tau}^2\sum_{i=1}^{k}\alpha_i^2+2L_{\tau}\|\x^*\|\sum_{i=1}^{k}\alpha_i. 
			\end{aligned}
\end{equation*}
Together with 
\begin{equation*}
2\sum_{i=1}^{k}\alpha_k(F_{\tau}^i-F_{\tau}^*)\geq 2(\underline{F}^{k}_{\tau}-F_{\tau}^*)\sum_{i=1}^{k}\alpha_k,
\end{equation*}
we have 
\begin{equation*}
\underline{F}^{k}_{\tau}-F_{\tau}^*\leq \frac{\|\x^1-\x^*\|^2}{2\sum_{i=1}^{k}\alpha_k}
+\frac{L_{\tau}^2\sum_{i=1}^{k}\alpha_i^2}{2\sum_{i=1}^{k}\alpha_k}+L_{\tau}\|\x^*\|  
\end{equation*}
and \eqref{eq3-42} is proved. 
\end{proof}

For choices of the stepsizes \eqref{stepsize1}, \eqref{stepsize2} and \eqref{stepsize3}, we have the following convergence results.	
\begin{corollary}
Under the assumptions of Theorem \ref{th3-1},  
\begin{itemize}
			\item[(1)] if stepsize $\alpha_k$ satisfies \eqref{stepsize1}, then 
			\begin{equation*}
				\underline{F}^{k}_{\tau}-F_{\tau}^*\leq \frac{\epsilon^2}{2k\alpha}
				+\frac{\alpha L_{\tau}^2}{2}+L_{\tau}\|\x^*\|;
			\end{equation*}
			\item[(2)] if stepsize $\alpha_k$ satisfies \eqref{stepsize2}, then 
			\begin{equation*}
				\underline{F}^{k}_{\tau}-F_{\tau}^*\leq\frac{\epsilon^2L_{\tau}}{2ks}
				+\frac{sL_{\tau}}{2}+L_{\tau}\|\x^*\|;
			\end{equation*}
			\item[(3)] if stepsize $\alpha_k$ satisfies \eqref{stepsize3}, then 
			\begin{equation*}
				\varlimsup_{k\to \infty} \underline{F}^{k}_{\tau}-F_{\tau}^*\leq L_{\tau}\|\x^*\|. 
\end{equation*}
\end{itemize}
\end{corollary}
	
\begin{theorem}
Under the assumptions of Theorem \ref{th3-1},  if stepsize $\alpha_k$ satisfies \eqref{stepsize4}, we have 
\begin{enumerate}
			\item[(1)] The sequence $\{ \|\x^k-\x^*\|^2\}$ converges;
			\item[(2)] For any given $\delta>0$, there exists $k_0$ such that $F_{\tau}^{k_0}\leq \rho+\delta$; 
			\item[(3)] The sequence $\{\x^k\}$ converges to some $\tilde{\x}\in S$ and sequence $\{F_{\tau}^{k}\}$ converges to $\rho\geq F_{\tau}(\tilde{\x})+L_{\tau}\|\x^*\|$. 
\end{enumerate}
\end{theorem}

\begin{proof}
		(1) From \eqref{le3-45} and \eqref{subdiff}, one has 
		\begin{equation}\label{le3-2}
			\begin{aligned}
				&\|\x^{k+1}-\x^*\|^2\\
				&\leq\|\x^k-\x^*\|^2+\alpha_k^2\|\d^k\|^2+2\alpha_k(L_{\tau}\|\x^*\|+F_{\tau}^*-F_{\tau}^k)\\
				&\makecell[c]{\overset{\eqref{stepsize4}}{=}}\|\x^k-\x^*\|^2+\frac{s_k^2(F_{\tau}^k-\rho)^2}{\|\d^k\|^2}+\frac{2s_k(F_{\tau}^k-\rho)(L_{\tau}\|\x^*\|+F_{\tau}^*-F_{\tau}^k)}{\|\d^k\|^2}\\
				&=\|\x^k-\x^*\|^2+s_k(F_{\tau}^k-\rho)\frac{s_k(F_{\tau}^k-\rho)-2(F_{\tau}^k-F_{\tau}^*-L_{\tau}\|\x^*\|)}{\|\d^k\|^2}\\
				&\makecell[c]{\overset{\eqref{stepsize41}}{\leq}}\|\x^k-\x^*\|^2+s_k(F_{\tau}^k-\rho)\frac{s_k(F_{\tau}^k-\rho)-2(F_{\tau}^k-\rho)}{\|\d^k\|^2}\\
				&= \|\x^k-\x^*\|^2+s_k(s_k-2)(F_{\tau}^k-\rho)^2/\|\d^k\|^2. 
			\end{aligned}
		\end{equation}
Since $s_k<2$, then the sequence $\{ \|\x^k-\x^*\|^2\}$ is nonincreasing and has lower bound, this implies $\{ \|\x^k-\x^*\|^2\}$ converges. 
		
(2) Assume on the contrary that $F_{\tau}^{k}> \rho+\delta$ for all $k$. 
From \eqref{le3-2}, we have 
\begin{equation*}
-s_k(s_k-2)(F_{\tau}^k-\rho)^2/\|\d^k\|^2\leq \|\x^k-\x^*\|^2-\|\x^{k+1}-\x^*\|^2. 
\end{equation*}
Note that $0<s_k\leq \beta<2$, $\|\d^k\|\leq L_{\tau}$,  $F_{\tau}^{k}-\rho>\delta$, then 
\begin{equation*}
s_k\beta\delta/L_{\tau}^2\leq \|\x^k-\x^*\|^2-\|\x^{k+1}-\x^*\|^2.
\end{equation*}
Summing $k$ from 1 to $\infty$ yields that 
\begin{equation*}
\beta\delta/L_{\tau}^2\sum_{i=1}^{\infty}s_i\leq  \|\x^1-\x^*\|^2,
\end{equation*}
which is a contradiction with $\sum_{i=1}^{\infty}s_i=\infty$. 
		
(3) It follows that $\{F_{\tau}^k\}$ is bounded, which means that there is a convergent subsequence of  $\{F_{\tau}^k\}$ and  there is some $K(0)$ such that $F_{\tau}^{K(0)}\leq \rho+1$. Base on $K(i)$, we can define 
\begin{equation*}
\Delta_i:=\min\{F_{\tau}^1,\cdots,F_{\tau}^{K(i)}\}-\rho>0, \quad \delta_i:=\min\{{1}/{2^{i+1}},\Delta_i/2\}.
\end{equation*}
There exist some $K(i+1)$ such that $F_{\tau}^{K(i+1)}\leq \rho+\delta_i$ and $K(i+1)>K(i)$. 
As constructed, $\{F_{\tau}^{K(i)}\}$ converges to $\rho$. In addition, $\x^{K(i)}$ is bounded, then a subsequence of $\x^{K(i)}$, say $\x^{K'(i)}$, converges to some point $\tilde{\x}$, which means  $\{F_{\tau}^{K'(i)}\}$ also converges to $\rho$ and $\tilde{\x}\in S$ from the closedness of $S$. Let $\d\in \partial F_{\tau}(\tilde{\x})$. 
Thus for every $i$, 
\begin{equation*}
F_{\tau}^{K'(i)}\geq F_{\tau}(\tilde{\x})+\langle \d,\x^{K'(i)}-\tilde{\x} \rangle, 
\end{equation*}
which implies that $\rho\geq F_{\tau}(\tilde{\x}) $. 
Define ancillary function $\tilde{F}_{\tau}(\x):=\max\{F_{\tau}(\x),\eta\}$. Note that $\tilde{F}$ is convex and finite on $S$ and $\tilde{\x}$ is a minimizer of $\tilde{F}$ on $S$. Furthermore, $\tilde{F}_{\tau}(\x)\geq F_{\tau}(\x)$ and each $F_{\tau}^k>\rho$, each $\d^k\in \partial F_{\tau}(\x^k)$ is a subgradient of $\partial \tilde{F}$ at $\x^k$. Thus, $F_{\tau}^k=\tilde{F}_{\tau}(\x^k)$ for all $k$, and $\{\x^k\}$ converges to $\tilde{\x}$. 
Since $F_{\tau}$ is continuous on $S$, then 
\begin{equation*}
\lim_{k\to \infty}F_{\tau}^k=F_{\tau}(\tilde{\x})=\lim_{i\to \infty}F_{\tau}^{K'(i)}=\rho. 
\end{equation*}
\end{proof}
	
Furthermore, if the objective function is strongly convex, we will have  stronger convergence results. A function $f$ is called $\sigma$-strongly convex ($\sigma>0$) if it satisfies the following inequality:
\begin{equation}\label{strconvex}
f(\y)\geq f(\x)+\langle \d,\y-\x\rangle+\frac{\sigma}{2}\|\y-\x\|^2, ~\forall \x,\y\in \R^n, 
\end{equation}
where $\d\in\partial f(\x)$. 
	
\begin{theorem}
Suppose  that Assumption \ref{ass3} holds and the objective function $F_{\tau}$ is $\sigma$-strongly convex. Let $\x^*$ be an optimal solution of problem \eqref{Pro2}. If  the initial point $\x^1$ satisfies $\|\x^1-\x^*\|\leq \epsilon$ for some $\epsilon>0$, then the sequence $\x^k$  by PSM with stepsize $\alpha_k$ with  $0<\alpha_k<\frac{1}{\sigma}$ satisfies
\begin{equation}\label{eq3-40}
			\begin{aligned}
				\|\x^{k+1}-\x^*\|^2&\leq \left(\|\x^1-\x^*\|^2-\frac{2L_{\tau}}{\sigma}\|\x^*\|\right)\prod_{i=1}^k(1-\sigma\alpha_i)\\
				&~~~~+L_{\tau}^2\sum_{t=1}^k\alpha_t^2 \prod_{i=t+1}^k(1-\sigma\alpha_i)+\frac{2L_{\tau}}{\sigma}\|\x^*\|.
\end{aligned}		
\end{equation}
Furthermore, we have 
\begin{equation}\label{3-41}
			\begin{aligned}
				\underline{F}_{\tau}^{k}-F_{\tau}^*&\leq\frac{\prod_{i=1}^k(1-\sigma\alpha_i)\epsilon^2}{2\sum_{i=1}^k\alpha_i\prod_{j=i+1}^k(1-\sigma\alpha_j)}\\
                &~~~~+\frac{\sum_{i=1}^k\left(L_{\tau}^2\alpha_i^2+2\alpha_iL_{\tau}\|\x^*\|^2\right)\prod_{j=i+1}^k(1-\sigma\alpha_j)}{2\sum_{i=1}^k\alpha_i\prod_{j=i+1}^k(1-\sigma\alpha_j)}.
			\end{aligned}
		\end{equation}
\end{theorem}

\begin{proof}
From \eqref{le3-21}, and combining $F_{\tau}$ being $\sigma$-strongly convex, we have		
		\begin{align}
			&\|\x^{k+1}-\x^*\|^2 \nonumber \\
			&\makecell[c]{\overset{\eqref{le3-21}}{\leq}}\|\x^k-\x^*\|^2+\alpha_k^2L_{\tau}^2+2\alpha_kL_{\tau}\|\x^*\|+2\alpha_k\langle \d^k,\x^*-\x^k \rangle \nonumber \\
			&\makecell[c]{\overset{\eqref{strconvex}}{\leq}}\|\x^k-\x^*\|^2+\alpha_k^2L_{\tau}^2+2\alpha_kL_{\tau}\|\x^*\|+2\alpha_k  \left( F_{\tau}^*- {F}_{\tau}^{k}-\frac{\sigma}{2}\|\x^*-\x^k\|^2\right)  \nonumber  \\ 
			&=(1-\sigma\alpha_k)\|\x^k-\x^*\|^2+\alpha_k^2L_{\tau}^2+2\alpha_kL_{\tau}\|\x^*\|
			+2\alpha_k  (F_{\tau}^*- {F}_{\tau}^{k}) \label{ineq3-32} 
			\\ 
            &\leq (1-\sigma\alpha_k)\|\x^k-\x^*\|^2+\alpha_k^2L_{\tau}^2+2\alpha_kL_{\tau}\|\x^*\|. \label{ineq3-40}
		\end{align}
Rewriting \eqref{ineq3-40}, we see that  
\begin{equation*}
\|\x^{k+1}-\x^*\|^2-\frac{2L_{\tau}}{\sigma}\|\x^*\|
		\leq (1-\sigma\alpha_k)\left(\|\x^k-\x^*\|^2-\frac{2L_{\tau}}{\sigma}\|\x^*\|\right)+\alpha_k^2L_{\tau}^2. 
\end{equation*}		
By induction on $k$, we have  
		\begin{equation*}
			\begin{aligned}
				&\|\x^{k+1}-\x^*\|^2-\frac{2L_{\tau}}{\sigma}\|\x^*\|\\
				&\leq \prod_{i=1}^k(1-\sigma\alpha_i)\left(\|\x^1-\x^*\|^2-\frac{2L_{\tau}}{\sigma}\|\x^*\|\right)
				+L_{\tau}^2\sum_{t=1}^k\alpha_t^2 \prod_{i=t+1}^k(1-\sigma\alpha_i).
			\end{aligned}
		\end{equation*}
where we define $\Pi_{i=t+1}^k(1-\sigma\alpha_i)=1$ when $t=k$, which means \eqref{eq3-40} holds. 
		
Rewriting \eqref{ineq3-32}, we have 
		\begin{equation*}
			\begin{aligned}
				\|\x^{k+1}-\x^*\|^2
				\leq (1-\sigma\alpha_k)\|\x^k-\x^*\|^2+\alpha_k^2L_{\tau}^2+2\alpha_kL_{\tau}\|\x^*\|
				+2\alpha_k  (F_{\tau}^*- {F}_{\tau}^{k}).
			\end{aligned}
		\end{equation*}
By induction on $k$, we obtain
		\begin{equation}\label{eq3-43}
			\begin{aligned}
				&\|\x^{k+1}-\x^*\|^2\\
				&\leq  \prod_{i=1}^k(1-\sigma\alpha_i)\|\x^1-\x^*\|^2
				+\sum_{i=1}^k\left(L_{\tau}^2\alpha_i^2+2\alpha_iL_{\tau}\|\x^*\|\right) \prod_{j=i+1}^k(1-\sigma\alpha_j)\\
				&~~~~-2\sum_{i=1}^k\alpha_i(F_{\tau}^i-F_{\tau}^*)\prod_{j=i+1}^k(1-\sigma\alpha_j). 
			\end{aligned}
		\end{equation}
Rewrite \eqref{eq3-43} as 
		\begin{equation*}
			\begin{aligned}
				&2\sum_{i=1}^k\alpha_i(F_{\tau}^i-F_{\tau}^*)\prod_{j=i+1}^k(1-\sigma\alpha_j)\\
				&\leq  \prod_{i=1}^k(1-\sigma\alpha_i)\|\x^1-\x^*\|^2-\|\x^{k+1}-\x^*\|^2
				+\sum_{i=1}^k\left(L_{\tau}^2\alpha_i^2+2\alpha_iL_{\tau}\|\x^*\|^2\right)\prod_{j=i+1}^k(1-\sigma\alpha_j)\\
				&\leq   \prod_{i=1}^k(1-\sigma\alpha_i)\|\x^1-\x^*\|^2
				+\sum_{i=1}^k\left(L_{\tau}^2\alpha_i^2+2\alpha_iL_{\tau}\|\x^*\|^2\right)\prod_{j=i+1}^k(1-\sigma\alpha_j).
			\end{aligned}
		\end{equation*}
Together with 
		\begin{equation*}
			\underline{F}_{\tau}^{k}-F_{\tau}^*\leq F_{\tau}^i-F_{\tau}^*, 
		\end{equation*}
		we have \eqref{3-41} holds and the proof is completed.  
\end{proof}

\subsection{Stochastic projected subgradient method}

The stochastic version of PSM is presented in Algorithm \ref{SPSM}, where $\beta_k$  can be taken as fixed stepsize $\beta_k=\beta$ or  $\beta_k=\frac{2}{\sigma(k+1)}$. Below, we establish the convergence results corresponding to these two stepsize strategies.

\begin{algorithm}[ht!]
	\caption{Stochastic	projected subgradient method (SPSM) for problem \eqref{Pro2}} \label{SPSM}	
	\KwIn{Given $\x^0=0$, let $k= 0$.}
	\While{not converged}
	{
	(\textit{Step 1}) Pick $\x^{k+1}=\proj_S(\x^k-\beta_k \tilde{\d}^k),$ where  $\beta_k$ is the stepsize  and  $\tilde{\d}^k$ is the stochastic version of the subgradient $\d^k$ at $\x^k$, that is,  $\mathbb{E}(\tilde{\d}^k)=\d^k\in \partial F_{\tau}(\x^k)$, and 
            \begin{equation*}
\d^k=\nabla f(\x^k)+ {\tau}\sum_{i=1}^{m}\lambda_i^k\nabla g_i(\x^k)+{\tau}\sum_{j=1}^{p}\mu_j^k\nabla h_j(\x^k)
                \end{equation*}
with $\lambda_i^k$ and $\mu_i^k$ are defined in \eqref{lambdamu}.
                
   (\textit{Step 2}) Check convergence: if 
\begin{equation*}
\|\x^{k+1}-\x^k\|\leq \epsilon,
 \end{equation*}  
 then stop. Otherwise, let $k \leftarrow  k+1$ and  go to \textit{Step 1}.		
	}	
	\KwOut{$\mathbf{x}$.}	
\end{algorithm}

	\begin{theorem}
		Let $\x^*$ and $\x^1$ be an optimal solution of problem \eqref{Pro2} and an initial point satisfying 
		$\|\x^1-\x^*\|\leq \epsilon$ with $\epsilon>0$, respectively. 
		Suppose that $\tilde{\d}^k$ satisfies the following conditions
		\begin{equation}\label{grad2}
			\mathbb{E}(\|\tilde{\d}^k\|)\leq M_1, \ \ \mathbb{E}(\|\tilde{\d}^k\|^2)\leq M_2. 
		\end{equation}
		Then the sequence $\x^k$ generated by SPSM with $\beta_k=\beta$ satisfies
		\begin{equation*}
			\mathbb{E}\left[F_{\tau}\left(\frac{1}{k}\sum_{i=1}^{k}\x^i\right)-F_{\tau}^*\right]\leq 
			\frac{\epsilon^2}{2\beta k} +\frac{1}{2}\beta M_2+M_1\|\x^*\|.
		\end{equation*}
	\end{theorem}
    
	\begin{proof}
		Using the similar proof as \eqref{le3-42}, we have 
		\begin{equation*}\label{th3-8}
			\begin{aligned}
				&\mathbb{E}\left(\|\x^{k+1}-\x^*\|^2\big|\x^k\right)\\
				&\leq\|\x^k-\x^*\|^2+\beta_k^2\mathbb{E}\left(\|\tilde{\d}^k\|^2\big|\x^k\right)
				+2\beta_k\|\x^*\|\mathbb{E}\left(\|\tilde{\d}^k\|\big|\x^k\right)
				+2\beta_k\mathbb{E}\left(F_{\tau}^*-F_{\tau}^k\big|\x^k\right),
			\end{aligned}
		\end{equation*}
		where $\mathbb{E}\left(\cdot \big|\x^k\right)$ is the expectation given $\x^k$. 
		Now taking expectation with respect to $\x^k$ on both sides, and by \eqref{grad2}, it follows that 
		\begin{equation*}\label{th3-7}
			\begin{aligned}
				&\mathbb{E}\left(\|\x^{k+1}-\x^*\|^2\right)\\
				&\leq\mathbb{E}\left( \|\x^k-\x^*\|^2\right) +\beta_k^2\mathbb{E}\left(\|\tilde{\d}^k\|^2 \right)
				+2\beta_k\|\x^*\|\mathbb{E}\left(\|\tilde{\d}^k\|\right)
				+2\beta_k\mathbb{E}\left(F_{\tau}^*-F_{\tau}^k\right) \\
				&\leq\mathbb{E}\left( \|\x^k-\x^*\|^2\right) +\beta_k^2M_2
				+2\beta_k\|\x^*\|M_1
				+2\beta_k\mathbb{E}\left(F_{\tau}^*-F_{\tau}^k\right).
			\end{aligned}
		\end{equation*}
		Thus, 
		\begin{equation*}
			\begin{aligned}
				2\beta_k\mathbb{E}( F_{\tau}^k-F_{\tau}^*)\leq\mathbb{E}( \|\x^k-\x^*\|^2) -\mathbb{E}(\|\x^{k+1}-\x^*\|^2)+\beta_k^2M_2
				+2\beta_k\|\x^*\|M_1.
			\end{aligned}
		\end{equation*}		
		Summing both sides of the above inequality form $k=1$ to $k$ and by $\beta_k=\beta$, we obtain 
		\begin{equation*}
			\begin{aligned}
				2\beta\sum_{i=1}^{k}\mathbb{E}( F_{\tau}^i-F_{\tau}^*)
				&\leq\|\x^1-\x^*\|^2 +k\left(\beta^2M_2
				+2\beta\|\x^*\|M_1\right)\\
				&\leq\epsilon^2 +k\left(\beta^2M_2+2\beta\|\x^*\|M_1\right).
			\end{aligned}
		\end{equation*}
		From Jensen's inequality, we have 
		\begin{equation*}
			\begin{aligned}
				\mathbb{E}\left[F_{\tau}\left(\frac{1}{k}\sum_{i=1}^{k}\x^i\right)-F_{\tau}^*\right]
				&\leq\frac{1}{k}\sum_{i=1}^{k}\mathbb{E}\left( F_{\tau}(\x^i)-F_{\tau}^*\right)\\
				&\leq\frac{\epsilon^2}{2\beta k} +\frac{1}{2}\left(\beta M_2+2M_1\|\x^*\|\right).
			\end{aligned}
		\end{equation*}
		The proof is finished. 
\end{proof}
	
If the objective function $F_{\tau}$ is strongly convex and the stepsize $\beta_k=\frac{2}{\sigma(k+1)}$, we can obtain a stronger convergence result of SPSM. 

	\begin{theorem}
		Let $\x^*$ and $\x^1$ be an optimal solution of problem \eqref{Pro2} and an initial point satisfying 
		$\|\x^1-\x^*\|\leq \epsilon$ with $\epsilon>0$, respectively. 
		Suppose that the objective function $F_{\tau}$ is $\sigma$-strongly convex and $\tilde{\d}^k$ satisfies the conditions given in \eqref{grad2}. 
		Then the sequence $\{\x^k\}$ generated by SPSM with $\beta_k=\frac{2}{\sigma(k+1)}$ satisfies
		\begin{equation}\label{th3-11}
			\mathbb{E}\left(\|\x^{k+1}-\x^*\|^2\right)\leq\frac{4M_2}{\sigma^2(k+1)}+\frac{2M_1\|\x^*\|}{\sigma}.
		\end{equation}
		Furthermore, it holds that
		\begin{equation*}
			\mathbb{E}\left[F_{\tau}\left(\sum_{i=1}^{k}\frac{2i}{k(k+1)}\x^i\right)-F_{\tau}^*\right]\leq 
			M_1\|\x^*\|+\frac{2M_2}{\sigma(k+1)}.
		\end{equation*}
	\end{theorem}
    
	\begin{proof}
		Using the similar proof as in \eqref{le3-2}, we have 
		\begin{equation*}\label{th3-10}
			\begin{aligned}
				&\mathbb{E}\left(\|\x^{k+1}-\x^*\|^2\big|\x^k\right)\\
				&\leq \|\x^k-\x^*\|^2+\beta_k^2\mathbb{E}\left(\|\tilde{\d}^k\|^2 \big|\x^k\right)
				+2\beta_k \|\x^*\|\mathbb{E}\left(\|\tilde{\d}^k\| \big|\x^k\right)+
				2\beta_k  \langle \tilde{\d}^k,\x^*-\x^k \rangle. 
			\end{aligned}
		\end{equation*}
		
		Taking expectation on both sides with respect to $\x^k$ of \eqref{th3-7} and considering that  
		$F_{\tau}(\x)$ is $\sigma$-strongly convex and $\beta_k=\frac{2}{\sigma(k+1)}$, we have  
		\begin{equation}\label{th3-7}
			\begin{aligned}
				&\mathbb{E}\left(\|\x^{k+1}-\x^*\|^2\right)\\
				&\leq\left(1-\sigma\beta_k\right)\mathbb{E}\left( \|\x^k-\x^*\|^2\right) +\beta_k^2M_2+2\beta_k \|\x^*\|M_1+
				2\beta_k\mathbb{E}\left( F_{\tau}^*-F_{\tau}(\x^k)\right)\\
				&=\frac{k-1}{k+1}\mathbb{E}\left( \|\x^k-\x^*\|^2\right) +\frac{4M_2}{\sigma^2(k+1)^2}+\frac{4M_1\|\x^*\|}{\sigma(k+1)} +
				\frac{4}{\sigma(k+1)}\mathbb{E}\left( F_{\tau}^*-F_{\tau}(\x^k)\right)\\
				&\leq \frac{k-1}{k+1}\mathbb{E}\left( \|\x^k-\x^*\|^2\right) +\frac{4M_2}{\sigma^2(k+1)^2}+\frac{4M_1\|\x^*\|}{\sigma(k+1)}.
			\end{aligned}
		\end{equation}
		
		Multiplying both sides of the above inequality by $k(k+1)$, yields that 
		\begin{equation*}
			\begin{aligned}
				0&\leq k(k-1)\mathbb{E}\left( \|\x^k-\x^*\|^2\right)-k(k+1)\mathbb{E}\left(\|\x^{k+1}-\x^*\|^2\right)+
				\frac{4M_2k}{\sigma^2(k+1)}+\frac{4kM_1\|\x^*\|}{\sigma}\\
				&\leq k(k-1)\mathbb{E}\left( \|\x^k-\x^*\|^2\right)-k(k+1)\mathbb{E}\left(\|\x^{k+1}-\x^*\|^2\right)+
				\frac{4M_2}{\sigma^2}+\frac{4kM_1\|\x^*\|}{\sigma}.
			\end{aligned}
		\end{equation*}
		Summing the above inequality from $1$ to $k$ implies that 
		\begin{equation*}
			k(k+1)\mathbb{E}\left(\|\x^{k+1}-\x^*\|^2\right)\leq \frac{4kM_2}{\sigma^2}+\frac{2k(k+1)M_1\|\x^*\|}{\sigma}, 
		\end{equation*}
		which means \eqref{th3-11} holds.  
		
		From \eqref{th3-7}, we have 
		\begin{equation*}
			\begin{aligned}
				&\mathbb{E}\left(\|\x^{k+1}-\x^*\|^2\right)\\
				&\leq\frac{k-1}{k+1}\mathbb{E}\left( \|\x^k-\x^*\|^2\right) +\frac{4M_2}{\sigma^2(k+1)^2}+\frac{4M_1\|\x^*\|}{\sigma(k+1)} +
				\frac{4}{\sigma(k+1)}\mathbb{E}\left( F_{\tau}^*-F_{\tau}(\x^k)\right).
			\end{aligned}
		\end{equation*}
		Multiplying both sides of the above inequality by $k(k+1)$ and reformulating it, we have 
		\begin{equation*}
			\begin{aligned}
				&\frac{4k}{\sigma}\mathbb{E}\left(F_{\tau}(\x^k)- F_{\tau}^*\right)\\
				&\leq  k(k-1)\mathbb{E}\left( \|\x^k-\x^*\|^2\right)-k(k+1)\mathbb{E}\left(\|\x^{k+1}-\x^*\|^2\right)+
				\frac{4M_2k}{\sigma^2(k+1)}+\frac{4kM_1\|\x^*\|}{\sigma}\\
				&\leq k(k-1)\mathbb{E}\left( \|\x^k-\x^*\|^2\right)-k(k+1)\mathbb{E}\left(\|\x^{k+1}-\x^*\|^2\right)+
				\frac{4M_2}{\sigma^2}+\frac{4kM_1\|\x^*\|}{\sigma}.
			\end{aligned}
		\end{equation*}
		Summing the above inequality from $1$ to $k$ and using Jensen's inequality yields that 
		\begin{equation*}
			\mathbb{E}\left[F_{\tau}\left(\sum_{i=1}^{k}\frac{2i}{k(k+1)}\x^i\right)-F_{\tau}^*\right]\leq 
			\frac{2M_2}{\sigma(k+1)}+M_1\|\x^*\|.
		\end{equation*}
The proof is finished. \end{proof}

\section{Conclusion}	\label{Section6}
In this paper, we investigate the exact penalization theory and derive the new bounds for cardinality-constrained optimization (CCO) problems subject to general equality and inequality constraints. In particular, under the CC-MFCQ and CC-pseudonormality, we establish the optimality conditions and local exact penalization properties. To solve the resulting exact penalization problem, we develop the projected subgradient method and the stochastic projected subgradient method. Our analysis yields  tighter convergence  tailored to the penalized cardinality structure. 

In the future, we are interested in exploring the local exact penalization results under some  weaker constraint qualifications than pseudonormality. Besides, we will design more efficient algorithms to handle the exact penalization problem with fast convergence rates.

\paragraph{Acknowledgement}

This work was supported in part by the National Natural Science Foundation of China (Nos. 12271309, 12371306).

\bibliographystyle{elsarticle-num}
\bibliography{mybibfile}

@article{bruckstein2009sparse,
	title={From sparse solutions of systems of equations to sparse modeling of signals and images},
	author={Bruckstein, Alfred M and Donoho, David L and Elad, Michael},
	journal={SIAM Review},
	volume={51},
	number={1},
	pages={34--81},
	year={2009},
	publisher={SIAM}
}

@article{wang2013multi,
  title={Multi-parameter {T}ikhonov regularization with the $\ell_0$ sparsity constraint},
  author={Wang, Wei and Lu, Shuai and Mao, Heng and Cheng, Jin},
  journal={Inverse Problems},
  volume={29},
  number={6},
  pages={065018},
  year={2013},
  publisher={IOP Publishing}
}

@article{zou2018selective,
	title={A selective overview of sparse principal component analysis},
	author={Zou, Hui and Xue, Lingzhou},
	journal={Proceedings of the IEEE},
	volume={106},
	number={8},
	pages={1311--1320},
	year={2018},
	publisher={IEEE}
}

@article{hazimeh2023grouped,
  title={Grouped variable selection with discrete optimization: Computational and statistical perspectives},
  author={Hazimeh, Hussein and Mazumder, Rahul and Radchenko, Peter},
  journal={The Annals of Statistics},
  volume={51},
  number={1},
  pages={1--32},
  year={2023},
  publisher={Institute of Mathematical Statistics}
}

@article{bertsimas2022scalable,
	title={A scalable algorithm for sparse portfolio selection},
	author={Bertsimas, Dimitris and Cory-Wright, Ryan},
	journal={INFORMS Journal on Computing},
	volume={34},
	number={3},
	pages={1489--1511},
	year={2022},
	publisher={INFORMS}
}

@article{anis2025end,
  title={End-to-end, decision-based, cardinality-constrained portfolio optimization},
  author={Anis, Hassan T and Kwon, Roy H},
  journal={European Journal of Operational Research},
  volume={320},
  number={3},
  pages={739--753},
  year={2025},
  publisher={Elsevier}
}

@article{he2023structured,
  title={Structured pruning for deep convolutional neural networks: A survey},
  author={He, Yang and Xiao, Lingao},
  journal={IEEE Transactions on Pattern Analysis and Machine Intelligence},
  volume={46},
  number={5},
  pages={2900--2919},
  year={2024},
  publisher={IEEE}
}

@inproceedings{benbaki2023fast,
  title={Fast as chita: Neural network pruning with combinatorial optimization},
  author={Benbaki, Riade and Chen, Wenyu and Meng, Xiang and Hazimeh, Hussein and Ponomareva, Natalia and Zhao, Zhe and Mazumder, Rahul},
  booktitle={International Conference on Machine Learning},
  pages={2031--2049},
  year={2023},
  organization={PMLR}
}

@article{tillmann2024cardinality,
	title={Cardinality minimization, constraints, and regularization: A survey},
	author={Tillmann, Andreas M and Bienstock, Daniel and Lodi, Andrea and Schwartz, Alexandra},
	journal={SIAM Review},
	volume={66},
	number={3},
	pages={403--477},
	year={2024},
	publisher={SIAM}
}

@article{bonami2009exact,
  title={An exact solution approach for portfolio optimization problems under stochastic and integer constraints},
  author={Bonami, Pierre and Lejeune, Miguel A},
  journal={Operations Research},
  volume={57},
  number={3},
  pages={650--670},
  year={2009},
  publisher={INFORMS}
}

@article{liang2024relaxed,
  title={Relaxed method for optimization problems with cardinality constraints},
  author={Liang, Yan-Chao and Lin, Gui-Hua},
  journal={Journal of Global Optimization},
  volume={88},
  number={2},
  pages={359--375},
  year={2024},
  publisher={Springer}
}

@article{burdakov2016mathematical,
	title={Mathematical programs with cardinality constraints: reformulation by complementarity-type constraints and a regularization method},
	author={Burdakov, Oleg and Kanzow, Christian and Schwartz, Alexandra},
	journal={SIAM Journal on Optimization},
	volume={26},
	number={1},
	pages={397--425},
	year={2016},
	publisher={SIAM}
}

@article{vcervinka2016constraint,
	title={Constraint qualifications and optimality conditions for optimization problems with cardinality constraints},
	author={{\v{C}}ervinka, Michal and Kanzow, Christian and Schwartz, Alexandra},
	journal={Mathematical Programming},
	volume={160},
	pages={353--377},
	year={2016},
	publisher={Springer}
}

@article{kanzow2021augmented,
	title={An augmented {L}agrangian method for cardinality-constrained optimization problems},
	author={Kanzow, Christian and Raharja, Andreas B and Schwartz, Alexandra},
	journal={Journal of Optimization Theory and Applications},
	volume={189},
	number={3},
	pages={793--813},
	year={2021},
	publisher={Springer}
}

@article{kanzow2025sparse,
  title={The sparse (st) optimization problem: reformulations, optimality, stationarity, and numerical results},
  author={Kanzow, Christian and Schwartz, Alexandra and Wei{\ss}, Felix},
  journal={Computational Optimization and Applications},
  volume={90},
  number={1},
  pages={77--112},
  year={2025},
  publisher={Springer}
}

@article{pan2017convergent,
	title={A convergent iterative hard thresholding for nonnegative sparsity optimization},
	author={Pan, Lili and Zhou, Shenglong and Xiu, Naihua and Qi, Hou-Duo},
	journal={Pacific Journal of Optimization},
	volume={13},
	number={2},
	pages={325--353},
	year={2017}
}

@article{beck2013sparsity,
	title={Sparsity constrained nonlinear optimization: Optimality conditions and algorithms},
	author={Beck, Amir and Eldar, Yonina C},
	journal={SIAM Journal on Optimization},
	volume={23},
	number={3},
	pages={1480--1509},
	year={2013},
	publisher={SIAM}
}

@article{beck2015minimization,
	title={On the minimization over sparse symmetric sets: projections, optimality conditions, and algorithms},
	author={Beck, Amir and Hallak, Nadav},
	journal={Mathematics of Operations Research},
	volume={41},
	number={1},
	pages={196--223},
	year={2015},
	publisher={INFORMS}
}

@article{pan2017restricted,
	title={Restricted robinson constraint qualification and optimality for cardinality-constrained cone programming},
	author={Pan, Lili and Luo, Ziyan and Xiu, Naihua},
	journal={Journal of Optimization Theory and Applications},
	volume={175},
	number={1},
	pages={104--118},
	year={2017},
	publisher={Springer}
}

@article{zhao2022lagrange,
	title={A Lagrange-Newton algorithm for sparse nonlinear programming},
	author={Zhao, Chen and Xiu, Naihua and Qi, Houduo and Luo, Ziyan},
	journal={Mathematical Programming},
	volume={195},
	number={1},
	pages={903--928},
	year={2022},
	publisher={Springer}
}

@article{kan2024second,
	title={Second-Order Conditions for the Existence of Augmented {L}agrange Multipliers for Sparse Optimization},
	author={Kan, Chao and Song, Wen},
	journal={Journal of Optimization Theory and Applications},
	volume={201},
	number={1},
	pages={103--129},
	year={2024},
	publisher={Springer}
}

@article{li2025sparse,
	title={Sparse Quadratically Constrained Quadratic Programming via Semismooth Newton Method},
	author={Li, Shuai and Zhou, Shenglong and Luo, Ziyan},
	journal={arXiv preprint arXiv:2503.15109},
	year={2025}
}

@book{nocedal2006numerical,
  title={Numerical Optimization},
  author={Nocedal, Jorge and Wright, Stephen J},
  year={2006},
  publisher={Springer}
}

@article{chen2016penalty,
	title={Penalty methods for a class of non-{Lipschitz} optimization problems},
	author={Chen, Xiaojun and Lu, Zhaosong and Pong, Ting Kei},
	journal={SIAM Journal on Optimization},
	volume={26},
	number={3},
	pages={1465--1492},
	year={2016},
	publisher={SIAM}
}

@article{guo2018necessary,
  title={Necessary optimality conditions and exact penalization for non-Lipschitz nonlinear programs},
  author={Guo, Lei and Ye, Jane J},
  journal={Mathematical Programming},
  volume={168},
  number={1},
  pages={571--598},
  year={2018},
  publisher={Springer}
}

@article{xiao2024optimality,
  title={Optimality conditions and constraint qualifications for cardinality constrained optimization problems},
  author={Xiao, Zhuoyu and Jane, J Ye},
  journal={Numerical Algebra, Control and Optimization},
  volume={14},
  number={3},
  pages={614--635},
  year={2024},
  publisher={Numerical Algebra, Control and Optimization}
}

@article{jacques2013robust,
	title={Robust 1-bit compressive sensing via binary stable embeddings of sparse vectors},
	author={Jacques, Laurent and Laska, Jason N and Boufounos, Petros T and Baraniuk, Richard G},
	journal={IEEE Transactions on Information Theory},
	volume={59},
	number={4},
	pages={2082--2102},
	year={2013},
	publisher={IEEE}
}

@article{liu2019one,
	title={One-bit compressive sensing with projected subgradient method under sparsity constraints},
	author={Liu, Dekai and Li, Song and Shen, Yi},
	journal={IEEE Transactions on Information Theory},
	volume={65},
	number={10},
	pages={6650--6663},
	year={2019},
	publisher={IEEE}
}

@article{li2023adaptive,
	title={Adaptive iterative hard thresholding for least absolute deviation problems with sparsity constraints},
	author={Li, Song and Liu, Dekai and Shen, Yi},
	journal={Journal of Fourier Analysis and Applications},
	volume={29},
	number={1},
	pages={5},
	year={2023},
	publisher={Springer}
}

@article{nguyen2017linear,
	title={Linear convergence of stochastic iterative greedy algorithms with sparse constraints},
	author={Nguyen, Nam and Needell, Deanna and Woolf, Tina},
	journal={IEEE Transactions on Information Theory},
	volume={63},
	number={11},
	pages={6869--6895},
	year={2017},
	publisher={IEEE}
}

@article{zhou2018efficient,
	title={Efficient stochastic gradient hard thresholding},
	author={Zhou, Pan and Yuan, Xiaotong and Feng, Jiashi},
	journal={Advances in Neural Information Processing Systems},
	volume={31},
	year={2018},
	pages={1988--1997},
}

@article{shang2021efficient,
	title={Efficient gradient support pursuit with less hard thresholding for cardinality-constrained learning},
	author={Shang, Fanhua and Wei, Bingkun and Liu, Hongying and Liu, Yuanyuan and Zhou, Pan and Gong, Maoguo},
	journal={IEEE Transactions on Neural Networks and Learning Systems},
	volume={33},
	number={12},
	pages={7806--7817},
	year={2021},
	publisher={IEEE}
}

@article{li2024stochastic,
	title={Stochastic {IHT} with stochastic {P}olyak step-size for sparse signal recovery},
	author={Li, Changhao and Ma, Zhixin and Sun, Dazhi and Zhang, Guoming and Wen, Jinming},
	journal={IEEE Signal Processing Letters},
	year={2024},
	publisher={IEEE}
}

@article{alber1998projected,
	title={On the projected subgradient method for nonsmooth convex optimization in a {H}ilbert space},
	author={Alber, Ya I and Iusem, Alfredo N and Solodov, Mikhail V},
	journal={Mathematical Programming},
	volume={81},
	pages={23--35},
	year={1998},
	publisher={Springer}
}

@article{auslender2009projected,
	title={Projected subgradient methods with non-{E}uclidean distances for non-differentiable convex minimization and variational inequalities},
	author={Auslender, Alfred and Teboulle, Marc},
	journal={Mathematical Programming},
	volume={120},
	pages={27--48},
	year={2009},
	publisher={Springer}
}

@book{rockafellar1998variational,
	title={Variational Analysis},
	author={Rockafellar, R Tyrrell and Wets, Roger JB},
	year={1998},
	publisher={Springer}
}

@article{pan2017optimality,
  title={Optimality conditions for sparse nonlinear programming},
  author={Pan, LiLi and Xiu, NaiHua and Fan, Jun},
  journal={Science China Mathematics},
  volume={60},
  number={5},
  pages={759--776},
  year={2017},
  publisher={Springer}
}

@article{andreani2012relaxed,
  title={A relaxed constant positive linear dependence constraint qualification and applications},
  author={Andreani, Roberto and Haeser, Gabriel and Schuverdt, Mar{\'\i}a Laura and Silva, Paulo JS},
  journal={Mathematical Programming},
  volume={135},
  number={1},
  pages={255--273},
  year={2012},
  publisher={Springer}
}

@article{guo2014new,
  title={New results on constraint qualifications for nonlinear extremum problems and extensions},
  author={Guo, Lei and Zhang, Jin and Lin, Gui-Hua},
  journal={Journal of Optimization Theory and Applications},
  volume={163},
  number={3},
  pages={737--754},
  year={2014},
  publisher={Springer}
}

@article{bertsekas2002pseudonormality,
  title={Pseudonormality and a {L}agrange multiplier theory for constrained optimization},
  author={Bertsekas, Dimitri P and Ozdaglar, Asuman E},
  journal={Journal of Optimization Theory and Applications},
  volume={114},
  number={2},
  pages={287--343},
  year={2002},
  publisher={Springer}
}

@article{allen1987generalization,
	title={A generalization of {P}olyak's convergence result for subgradient optimization},
	author={Allen, Ellen and Helgason, Richard and Kennington, Jeffery and Shetty, Bala},
	journal={Mathematical Programming},
	volume={37},
	pages={309--317},
	year={1987},
	publisher={Springer}
}

@article{boyd2003subgradient,
  title={Subgradient methods},
  author={Boyd, Stephen and Xiao, Lin and Mutapcic, Almir},
  journal={lecture notes of EE392o, Stanford University, Autumn Quarter},
  volume={2004},
  number={01},
  pages={175},
  year={2003}
}

\end{document}